\documentclass[12pt, reqno]{amsart}
\usepackage{verbatim}
\usepackage{graphicx}
\usepackage{color}
\usepackage{enumerate}
\usepackage{epstopdf}

\usepackage{accents}

\newtheorem{theorem}{Theorem}
\newtheorem{lemma}[theorem]{Lemma}
\newtheorem{proposition}[theorem]{Proposition}
\newtheorem{corollary}[theorem]{Corollary}
\newtheorem{conjecture}{Conjecture}
\newtheorem{question}[conjecture]{Question}

\newtheorem*{claim*}{Claim}

\theoremstyle{definition}
\newtheorem*{definition*}{Definition}
\newtheorem{remark}[theorem]{Remark}
\newtheorem{example}[theorem]{Example}

\newcommand\ns[1]{ \left\{ {#1} \right\} }

\newcommand{\Z}{{\mathbb Z}}
\newcommand{\R}{{\mathbb R}}
\newcommand{\N}{{\mathbb N}}

\newcommand{\pM}{{\mathbb{M}}}

\newcommand{\eqd}{\stackrel{d}{=}}

\newcommand{\PP}{P}

\newcommand\garbage[1]{}

\renewcommand{\P}{{\mathbb P}}

\newcommand{\dff}[1]{\textbf{\emph{#1}}}

\newcommand{\ronn}{|}

\newcommand{\erk}{\hfill \ensuremath{\Diamond}} 

\newcommand{\bb}[1]{\mathbb{#1}}
\newcommand{\mc}[1]{\mathcal{#1}}

\newcommand{\ug}[1]{\mc S({#1})}
\newcommand{\usg}[1]{\mc S\text{*}({#1})}

\usepackage{bbm}

\newcommand\sk[1]{  \mathfrak{#1}}

\begin{document}
	
	\title[Finitary isomoprhisms]{Finitary isomorphisms of some infinite entropy Bernoulli flows}
	\author{Terry Soo}

	\address{Department of Mathematics,
		University of Kansas,
	405 Snow Hall,
		\indent 1460 Jayhawk Blvd,
		Lawrence, Kansas 66045-7594}
		\email{tsoo@ku.edu}
	\urladdr{http://www.terrysoo.com}
	\thanks{Funded in part by a General Research Fund.}
	
	
\begin{abstract}
A consequence of Ornstein theory is that the infinite entropy flows associated with Poisson processes and continuous-time irreducible Markov chains on a finite number of states are isomorphic as measure-preserving systems.   We give an elementary construction of such an  isomorphism that has an additional finitariness property, subject to the additional  conditions that the Markov chain has a uniform holding rate and a mixing skeleton.  
\end{abstract}
	
\maketitle

\section{Introduction}

A \dff{measure-preserving system}  is a probability space $(K, \kappa)$ endowed with a group of symmetries $T=(T_t)_{t \in G}$ such that $T_{t} \circ T_{s} = T_{t+s}$ and $\kappa \circ T_{s} = \kappa$ for all $t,s \in G$, where the group $G$ will usually be $\R$ or $\Z$.  In the case that $G = \R$, we also refer to the system as a \dff{flow}.    A system $(K', \kappa', T')$ is a \dff{factor} of  the system $(K, \kappa, T)$ if there exists a measurable mapping $\phi: K \to K'$ such that the push-forward of $\kappa$ under $\phi$ is $\kappa'$ and $\phi$ is \dff{equivariant} so that $\phi \circ T_t = T'_t \circ \phi$ for all $t \in G$; the map $\phi$ is also called a \dff{factor} map, and if $\phi$ is a bijection almost surely,  then its inverse map witnesses that $(K, \kappa, T)$ is a factor $(K', \kappa', T')$, and the two systems are \dff{isomorphic}.   We will be concerned with systems arising  from Poisson point processes and continuous-time  Markov chains.

Let $\Pi$ be a (homogeneous) Poisson point process on $\R$ with intensity $r >0$.  We will assume that $\Pi$ takes values on the space of 
Borel simple point measures $\pM$ 
with law $P_r(\cdot) = \P( \Pi \in \cdot)$.    We  endow  $\pM$ with $T=(T_t)_{t \in \R}$ the group of translations of $\R$   which act on $\pM$ via   $T_t(\mu) = \mu \cdot T_{-t}$  for $t \in \R$ and $\mu \in \pM$.   We refer to the measure-preserving system  $(\pM, P_r,   T)$ as a \dff{Poisson system} of intensity $r$.   
Let $X = (X_t)_{t \in \R}$ be a time-homogeneous  stationary continuous-time irreducible Markov chain with the  finite state space $A$.   The  Markov chain takes values on the space $A^{\R}$ which we also endow with the group of translations  by setting $(T_t x)_s = x_{s+t}$ for $x \in A^{\R}$ and $t,s \in \R$.   Let $L$ be the law of $X$.    We refer to the measure-preserving system $(A^{\R}, L, T)$ as a \dff{Markov system}.  
 A consequence of Ornstein theory \cite{MR3052869}  is the following theorem.

\begin{theorem}[Feldman and Smorodinsky \cite{MR0274718}]
\label{old}
Poisson systems are isomorphic to irreducible Markov systems with a finite number of states.
\end{theorem}

Theorem \ref{old} is a cumulation of many results of Ornstein and his collaborators, which we will briefly discuss in Section \ref{bern}.
We will prove the following version of Theorem \ref{old}.

   Note that if $(A^{\R}, L, T)$ is a Markov system, then  $L=L_{M, \lambda}$ is determined by the transition probabilities $M \in [0,1]^{A \times A}$ and the holding rates $\lambda =  (\lambda_a)_{a \in A}$.   If $X$ is irreducible, then the matrix $M$ is irreducible, but may not be aperiodic and  $M(a,a) =0$ for all $a \in A$.   
   Recall that a Markov chain in a state $a$  changes states in an exponential amount of time with mean $(\lambda_a)^{-1}$ and independently  selects the next state $b$ with probability $M(a,b)$.  We say that $X$ has a \dff{uniform holding rate} $s>0$ if $\lambda_a = s$ for all $a \in A$.

    Recall that the \dff{jump times} of a continuous-time Markov chain are the random exponential times where it changes states.  	Given a continuous-time Markov chain $X \in A^{\R}$ with jump times $(J_i)_{i \in \Z}$, where $J_i <J_j$ if $i < j$  and $J_0 \leq 0$ is taken to the jump time  nearest to the origin from the left,  we define $\sk{X} \in A^{\Z}$ to be the  \dff{skeleton} of $X$, where $\sk{X}_i = X_{J_i}$ for all $i \in \Z$.   The  skeleton is mixing if and only if $M$ is irreducible and aperiodic.    

For $\mu \in \pM$, we set  $\mu \ronn_I(\cdot) := \mu(\cdot \cap I)$  to be the restriction of $\mu$ to an interval $I \subseteq \R$.    
We say that $\mu, \mu' \in \pM$ \dff{agree} on an interval $I$ if their restrictions to $I$ are the same.    Similarly, for $x,x' \in A^{\R}$, we set $x \ronn_I \in \R^{I}$ to be given $x\ronn_I(s) = x(s)$ for all $s \in I$ and we say that $x$ and $x'$ \dff{agree} on an interval $I$ if there restrictions to $I$ are the same.  Let $\phi$ be a factor from $(\pM, P_r, T) $ to $(A^{\R}, L, T)$.        We say a \dff{coding window} of $\phi$ is a function $w:\bb M\rightarrow \bb N\cup \ns{ \infty }$ 
	such that if $\mu$ and $\mu'$ agree on ${(-w(\mu) ,w(\mu) )}$, 
	then
	%
$\phi(\mu)$ and $\phi(\mu')$ agree on $(-1,1)$.

	We say that $\phi$ is \dff{finitary} if there exists a coding window $w$ such that $w$ is finite $P_{r}$-almost surely.    
	Thus if $\phi$ is a finitary factor,
	 then we can determine  $\phi(\mu)$ on an interval given  $\mu$ restricted to a  finite window.         We will use the analogous definitions for a factor maps from Markov systems to Poisson systems.

 \newpage
 
\begin{theorem}
\label{new}
Any Poisson system 
is finitarily isomorphic to 
any  irreducible Markov system with: a finite number of states,  a uniform holding rate,  and a mixing skeleton; furthermore,  there exists a finitary isomorphism with a  finitary inverse. 
\end{theorem}

We will prove  Theorem \ref{new} by exploiting some recent results of Soo and Wilkens \cite{Sw2}; in particular they proved that any two Poisson systems are finitarily isomorphic.  Here, although the factor map $\phi$ is a self-map of $\pM$, the definition of finitary is exactly as stated above.

\begin{theorem}[Soo and Wilkens \cite{Sw2}]
	\label{SW}
Any two Poisson systems are finitarily isomorphic; furthermore,   there exists a finitary isomorphism with a  finitary inverse.
\end{theorem}

\section{Bernoulli systems}
\label{bern}

\subsection{A very brief summary of Ornstein theory}
 We discuss some of 
 the
 developments of Ornstein theory  leading up to Theorem \ref{old}.   We recommend the articles \cite{MR2342699, MR3052869,  MR0447525,  MR1023980,  MR0304616}  for more information. 
 Ornstein theory tells us that  in disguise, up to a change of coordinates, many random and seemingly deterministic and mechanical  
 systems of interest have the same behavior as sequences of 
 dice rolls.
 
  Recall that a \dff{Bernoulli shift} is a discrete time system $(A^{\Z}, \alpha^{\Z}, \sigma)$, where $\kappa$ is a probability measure on the finite or countable set $A$, and $\sigma$ is the \dff{left-shift} given by $(\sigma x)_i = x_{i+1}$.  The Kolmogorov-Sinai measure-theoretic entropy for a Bernoulli shift  is given by the Shannon entropy: $$-\sum_{a \in A}  \alpha_a \log (\alpha_a)$$ and is an isomorphism-invariant \cite{MR0103255, MR0103256}.    Ornstein \cite{Ornstein, Ornstein-infinite} proved  it is a complete isomorphism-invariant  for Bernoulli shifts; furthermore, Ornstein theory provides a systematic  method to check whether a system is \dff{Bernoulli}, that is  {\em isomorphic} to a Bernoulli shift.   Ornstein theory tells us that to check whether a system is Bernoulli, one only needs to verify that it satisfies 
  one of 
  various mixing-type conditions; usually the condition {\em very weak Bernoulli} is the easiest one to verify.    As a result of Ornstein theory a host of systems are known to be Bernoulli, including mixing finite state Markov chains  \cite{MR0274718}, ergodic toral automorphisms \cite{MR0294602}, and any factor of a 
  Bernoulli system \cite{MR0274717}.    
  
   We say that a flow $(K, \kappa, T)$ is \dff{Bernoulli} if for each time $t >0$, the discrete time system $(K, \kappa, (T_{nt})_{n \in \Z})$ is Bernoulli.  The entropy of a flow is given by the entropy of the time-one map $T_1$.  The analogous theory has also been developed for flows.   Bernoulli flows of the same entropy are isomorphic \cite{Orn-flow}.    Finite entropy Bernoulli flows include:  ergodic Totoki flow \cite{Orn-tot}, geodesic flow  on compact surface of negative curvature \cite{Orn-geo},  and Sinai billiards \cite{Orn-billard}.  In the infinite entropy case that concerns this article,  Bernoulli flows include: 	Poisson systems, irreducible Markov systems  (Feldman and Smorodinsky \cite{MR0280678}), and more generally mixing Markov shifts of kernel type, which include  Brownian motion on 
   bounded reflecting regions (Ornstein and Shields \cite{MR0322137}).

 \subsection{Finitary isomorphism of mixing Markov chains}
 
 The isomorphisms given by  Ornstein's theory are provided by an existence theorem (think Baire category \cite{BurKeaSer})    and will not in general have a finitary property.       Keane and Smorodinsky \cite{keanea, keaneb, keanec} proved using what is now sometimes referred to as a {\em marker-filler} method that mixing finite state Markov chains of the same entropy are isomorphic by explicitly constructing a  factor map that is finitary and has a 
 finitary 
 inverse; here  the factor $\phi : A^{\Z} \to A^{\Z}$ is \dff{finitary} if  $ \phi(x)_0 = \phi(x')_0$ provided that $x$ and $x'$ agree on the coordinates $(-w(x), w(x))$, where the coding window $w: A^{\Z} \to \N$ is finite almost-surely.   We will make use Keane and Smorodinsky's result in our proof of Theorem \ref{new}.  
 
  In an elaborate extension of  the marker-filler method, Rudolph \cite{rud-fin} has tried to develop a theory that parallels Ornstein's, which  outputs finitary isomorphisms instead of general measurable isomorphisms.    He has applied this theory to show that countable state finite entropy mixing  Markov chains that are {\em exponentially recurrent}   are finitarily Bernoulli \cite{MR684246}.   However, despite Rudolph's machinery,  the list of systems that  are finitarily Bernoulli remains much smaller than the list of Bernoulli systems.    It is our hope to expand the list of finitarily Bernoulli systems.

 \subsection{Isomorphisms of Poisson point processes on $\R^d$}
 
 Ornstein and Weiss \cite{MR910005} also proved a very general isomorphism theorem for Poisson point processes, which implies that any two Poisson point processes on $\R^d$ are isomorphic via an isometery-equivariant factor; the finitary version of this result was recently proved by Soo and Wilkens \cite{Sw2}.  Previously,  Kalikow and Weiss \cite{MR1143427} proved that with respect to the group of 
 integer 
 translations, Poisson point processes on $\R$ are  finitarily isomorphic and finitarily Bernoulli.

 \section{The Proof of Theorem \ref{new}}
      
 \subsection{An outline of the proof}
 
By Theorem \ref{SW}, without loss of generality we may assume that the  Poisson point process, we are to transform,   has the same intensity as the uniform holding rate of the desired Markov chain.    Thus Theorem \ref{SW} already gives us all the jump times  of the desired Markov chain; we are only missing  values for the (new) states at the jump times.       By employing  the isomorphism of Keane and Smorodinsky \cite{keanec}, it suffices to consider independent and identically distributed (i.i.d.)\ states.  Hence we need to prove the following finitary isomorphism concerning marked Poisson point processes; more specifically,  Poisson point processes where the points are assigned i.i.d.\ labels.

If $\mu \in \pM$,  we write $[\mu] = \ns{ p_i}_{i \in \Z} \subset \R$ to be the points of $\mu$, where our convention will be that  the points are ordered so that $p_i < p_j$ if $i <j$, and we take $p_0 \leq 0$ to be the point closest to the origin from the left.   Let $A$ be a set; we will mainly be concerned with the case that $A$ is finite or countable.   We will refer to elements of $x \in A^{\Z}$ as \dff{marks}.   Let $(\mu, x) \in \pM \times A^{\Z}$; we will think of the points of $\mu$ as marked by the corresponding term in the sequence $x$ so that if $[\mu] = \ns{p_i}_{i \in \Z}$, then
$$[(\mu, x)] = \ns{  (p_i, x_i) }_{i \in \Z}$$
is the sequence of \dff{marked points}.

   We will define translations of $\pM \times A$ which preserve the markings.    Let  $\theta$ be a translation of $\R$.    We set $\theta(\mu, x) = (\theta \mu, \sigma^n x)$, where $\sigma$ is the usual left-shift and $n \in \Z$ is such that $\theta (p_n) \leq 0$ is the point of $\theta \mu$ that is closest to the origin from the left.    Thus if the point $p_0$ receives the mark $x_0$, it will receive the same mark even after it is translated.     
   A mapping $\psi:  \pM \to \pM \times A^{\Z}$ is \dff{translation-equivariant} if $\psi \circ T_t = T_t \circ \psi$ for all $t \in \R$.

	\begin{proposition}[Marked Poisson point processes]
		\label{inter}
		Let $(A, \alpha)$ be a standard probability space.    There exists
		 $\psi:\pM \to \pM \times A^{\Z}$ with the following properties.   Write $\psi(\mu) = (\psi(\mu)_1, \psi(\mu)_2)$.      Let $ r>0$.  
		\begin{enumerate}[(a)]
			\item
			\label{prod}
			If $\Pi$ is a Poisson point process on $\R$ of intensity $r$, then $\psi(\Pi)_1$ is Poisson point process on  $\R$ of intensity $r$ and $\psi(\Pi)_2$ has law  $\alpha^{\Z}$. 
			\item
			\label{independence}
			The Poisson point process $\psi(\Pi)_1$ and the sequence of random variables $\psi(\Pi)_2$ are independent. 
	      \item
			\label{inj-a}
			On a set of $\PP_r$-full measure,  the map $\psi$ is injective.
			\item
The map $\psi$ is 		translation-equivariant.   
			\item
			\label{fin-ex}
			The map $\psi$ is finitary so that there is a $P_r$-almost surely finite $w: \pM \to \N$ such that if 
			$\mu$ and $\mu'$ agree on $(-w(\mu) ,w(\mu))$,
			then 
			$\psi(\mu)_1$ and $\psi(\mu')_1$ agree on $(-1,1)$, 
			where
			and all the markings of the points of $\psi(\mu)_1$ and $\psi(\mu')_1$ are the same for their points in $(-1,1)$.  
			\item
			The inverse map $\psi^{-1}$ is finitary so that there exists a $P_r \times \alpha^{\Z}$-almost surely finite $w: \pM \times A^{\Z} \to \N$ such that if 
			$\mu$ and $\mu'$ agree on $(-w(\mu,x) ,w(\mu,x))$ and $x$ and $x'$ also  agree on  $(-w(\mu,x) ,w(\mu,x))$, then 
%
$\psi^{-1} (\mu, x)$ and $\psi^{-1} (\mu', x')$ agree on $(-1,1)$.  
		\end{enumerate}
	\end{proposition}
	
	%
	%
	The proof of Proposition \ref{inter} will involve variations of some ideas from Soo and Wilkens \cite{Sw2},  Holroyd, Lyons, and Soo \cite{splitting}, and Ball \cite{MR2133893}.     
	
Using  Proposition \ref{inter}, it is not difficult to prove Theorem \ref{new}.   
If two random variables $W$ and $W'$ have the same law, then we write $W \eqd W'$.

	\begin{proof}[Proof of Theorem \ref{new}]
		Let $\Pi$ be a Poisson point process of intensity $s>0$.
		Let $X$ be a stationary continuous-time  irreducible  Markov chain with a  finite state space $B$ with a uniform holding rate $r>0$ and a mixing skeleton.    We will define a finitary isomorphism so that $\phi(\Pi) \eqd X$.    Without loss of generality, by  Theorem \ref{SW}, we may assume that $\Pi$ has intensity $s=r$.

		  Suppose that $X$ has transition probabilities given by $M \in [0,1] ^{B \times B}$, and let $\beta \in [0,1]^B$ be the stationary distribution for $M$, so that $\beta M = \beta$.  Let $\sk{X}$ be the skeleton of $X$, so that $\sk{X}$ is a discrete-time mixing Markov chain with  transition matrix $M$ and stationary distribution $\beta$.   Let $L'$ be the law of $\sk{X}$.    The entropy of the system $(B^{\Z}, L', \sigma)$ is given by 
		$$0< h= - \sum_{i,j \in B}  \beta_i M_{ij} \log (M_{ij}) < \infty.$$  Let $\alpha$ be a probability measure on a finite set $A$ such that 
		$$ -\sum_{a \in A}  \alpha_a \log (\alpha_a)= h.$$    
	Since we assume that the skeleton $\sk{X}$ is mixing,	by 	Keane and Smorodinsky \cite{keanec}, there exists a finitary isomorphism $\zeta:  A^{\Z} \to B^{\Z}$ of the systems $(A^{\Z}, \beta^{\Z}, \sigma)$ and $(B^{\Z}, L', \sigma)$.

				With the parameters: $A$ and $\alpha$, let  $\psi$ be the finitary isomorphism from Proposition \ref{inter}.  Consider the mapping 
				$$ \xi(\mu) = \big(\psi(\mu)_1, \zeta \circ \psi(\mu)_2 \big).$$   It is easy to verify that $\xi$ is translation-equivariant.   The finitary properties of $\xi$ are inherited from $\psi$ and $\zeta$ in the following way.
				From property \eqref{fin-ex}, we already know that $\xi$ is finitary in the first coordinate.  Write  $[\psi(\mu)_2] = \ns{ w_i}_{i\in \Z}$,     $[\psi(\mu)_1] = \ns{q_i}_{i\in \Z}$, and $[\zeta \circ \psi(\mu)_2] = \ns{ x_i}_{i \in \Z}$.     Suppose that $q_n, q_{n+1}, \ldots, q_m$ are the points of $\psi(\mu)_1$ in $[-1,1]$, with markings $x_n, \ldots, x_{m} \in B$.        Since the mapping $\zeta$ given to us by Keane and Smorodinsky is finitary and translation-equivariant, in order to determine
				 $x_n, \ldots, x_m$, we need only the values  $w_{-N}, \ldots, w_N \in A$ for some large $N = N(w)$.  Similarly, 
				  $\psi$  is
				  finitary and translation-equivariant, so  we also only need $\mu$ restricted to some large window to determine $w_{-N}, \ldots, w_N$.  Since Keane and Smorodinsky's map also has a finitary inverse, similar considerations give that $\xi$ has a finitary inverse.

				 By properties \eqref{prod} and \eqref{independence} of $\psi$  and the definition of $\xi$, we have    $\psi(\Pi)_1 \eqd \Pi$ and $\zeta \circ \psi(\Pi)_2 \eqd \sk{X}$ and that these two coordinates of $\xi(\Pi)$ are independent.   Another  important consequence of assuming that the continuous-time Markov chain has a uniform holding rate is that  the jumps times $(J_i)_{i \in \Z}$ of the continuous-time Markov chain $X$ are independent of its skeleton $\sk{X}$.  Thus
				 
				 \begin{equation}
				 	\label{jump-eq}
				  \xi(\Pi) \eqd \Big(\sum_{i \in \Z} \delta[J_i], \sk{X} \Big).
				  \end{equation}

 			It remains to piece together the required Markov chain from the jumps times and skeleton.  Recall that $[\psi(\mu)_1] = \ns{q_i}_{i \in \Z}$ and $[\zeta \circ \psi(\mu)_2] = \ns{ x_i}_{i \in \Z}$.     For $t \in \R$, let $q(t) = i$, where $i$ is the largest integer for which $q_i \leq t$.        Set 
				$$\phi(\mu)_t =  x_{q(t)}.$$
				By \eqref{jump-eq}, we have $\phi(\Pi) \eqd X$ and $\phi$ inherits all the the other required properties from $\xi$.
		\end{proof}

\begin{remark}
	Recall  Rudolph \cite{MR684246} proved that countable state finite entropy mixing Markov chains that are exponentially recurrent   are finitarily Bernoulli.  Thus the proof of Theorem \ref{new} extends immediately to the case where the skeleton of the continuous-time Markov chain satisfies the conditions of Rudolph's theorem. \erk
	\end{remark}

 \subsection{The Proof of Proposition \ref{inter}}
   
   We begin with 
   a 
   simple toy example that will motivate some of the technical tools we need to introduce.    We say that a random variable $U$ is \dff{uniformly distributed} on an interval $[a,b]$, if $\P(U \leq x) = \frac{x-a}{b-a}$ for all $x \in [a,b]$.  Let $\delta_s = \delta[s] \in \pM$ denote the Dirac point measure at $s\in \R$.  
\begin{example}[Generating a marked Poisson point process] 
	\label{trans-ex}   Suppose we are given $\Pi$, a Poisson point process on $\R$ of intensity $r$ and we want to generate in a bijective way a Poisson point process on $\R$ of intensity $r$, where every point of the process is assigned, in an independent way, a fair coin-flip.    Furthermore, we want to do this in a way that is translation-equivariant with respect to integer translations.   We proceed in the following way.

	Consider the partition of $\R$ given by $\ns{ [i, i+1) }_{i \in \Z}$. Call an interval \dff{special} if it contains exactly one point of $\Pi$; we will also refer to the point in a special interval as \dff{special}.    Let $(i_k)_{k \in \Z}$ be the right endpoints of the special intervals, where $i_0 \leq 0$ is chosen to the largest such endpoint.    We assign each point of $\Pi$ that is not in a special interval to the nearest special interval to the right of the point.
	   Let $(U_{i_k})_{k \in \Z}$ be the special points.   Suppose each special interval $[i_k -1, i_k)$ is assigned $n_k \geq 0$ non-special points.       By elementary properties of Poisson point processes, the random variables $(i_k - U_{i_k})_{k \in \Z}$ are i.i.d.\ and uniformly distributed on the unit interval $[0,1)$; furthermore, these random variables are independent of $\Pi$ with the special points removed:
	$$\Pi - \sum_{k \in \Z} \delta_{U_{i_k}}.$$   
By using binary expansions (see also Lemma \ref{borel}), it is easy to see that for each $k \in \Z^{+}$ there exists a  bijection   $f_k: [0,1) \to [0,1) \times \ns{0,1}^{n_k+1}$ such that if $U$ is uniformly distributed on the unit interval, then $f_k(U)_1$ is uniformly distributed on the unit interval and independent of $f_k(U)_2$, a sequence of $n_k+1$ independent fair coin-flips.    Hence 
$$\Pi':= \sum_{k \in \Z} \delta[{  i_k - f_k(i_k - U_{i_k})_1}] +    \Pi - \sum_{k \in \Z} \delta[U_{i_k}] \eqd \Pi$$
 is independent of all the sequences of independent coin-flips given by $ (f_k(i_k - U_{i_k})_2)_{k \in \Z}$.    Thus in each special interval we have, in a bijective way,  resampled the special point and generated enough independent fair coin-flips to mark every point assigned to it, including the resident special point.   \erk
	\end{example}   
\subsubsection{Some special tools}
In Example \ref{trans-ex}, we  make use of an obvious partition of $\R$ into interval of unit length, since we only require that the mapping be translation-equivariant with respect to unit translations.     We will make use the following construction of Holroyd, Lyons, and Soo \cite{splitting} which will give us as function of the Poisson point process a  fully translation-equivariant finitary partition of $\R$ with enough independence properties that will allow us to imitate the construction in Example \ref{trans-ex}.
  Let $\mc F$ be the set of closed subsets of $\R$.   
 
 	\begin{proposition}[Holroyd, Lyons, and Soo \cite{splitting}]
 		\label{key}  There exists a  mapping $\mc{S}: \pM \to \mc F $ such that for $\mu \in \pM$, we have that  $\mc{S}({\mu})$ is a disjoint union of closed intervals of length two,  which we  refer  as \dff{globes} under $\mu$; the globes have  the following properties.
 	 \begin{enumerate}[(a)]
 	\item
The mapping $\mc S : \pM \to \mc F$ is translation-equivariant  so that $\ug{ T_t\mu} = T_t \ug{\mu}$ for all $\mu \in \pM$ and all  $t \in \R$ .
 \item
 \label{finitary-globes}
For any $\mu, \mu' \in \pM$, if they agree on $ (\mc{S}( \mu))^c$, then $\mc{S}(\mu) = \mc{S}(\mu')$.   Furthermore,  for any $z \in \R$ any $\mu, \mu' \in \pM$, if $[z-1, z+1]$ is a globe under $\mu$, then whenever $\mu$ and $\mu'$ agree on $[z-126, z+126]$, then $[z-1, z+1]$ is also a globe under $\mu'$.
 \item
\label{distance}
 The distance between the endpoints of two distinct globes is greater than $50$.
 \item
If $\Pi$ is a Poisson point process on $\R$, then almost surely $\ug{\Pi}$ is an infinite union of globes.
\item
If $\Pi$ and $\Gamma$ are independent Poisson point processes on $\R$ of the same intensity, then the point process given by  $$\Lambda:=\Gamma|_{\ug{\Pi}} +\Pi|_{\ug{\Pi}^c}$$ has the same law as $\Pi$ and $\ug{\Pi}=\ug{\Lambda}$.
  \end{enumerate}
 	\end{proposition}
 \begin{remark}   The definition of $\mc S$ is explicit and elementary resulting in property \eqref{finitary-globes} which will imply that constructions involving $\mc S$ will have a finitary property.     Proposition \ref{key} also holds in higher dimensions where the intervals are replaced by  balls.    Similar ideas first  appeared in Ball \cite{MR2133893}, but the formulation is slightly different and  her construction uses stopping times and does not extend to higher dimensions.  \erk
 \end{remark}
 
  We refer to the map $\mc S$ as a \dff{selection rule}.    Let $\Pi$ be a Poisson point process on $\R$.   We call the globes that contain exactly one point of the point process \dff{special} and set $\usg{\Pi}$ to be the union of the special globes under $\Pi$.  	So that we may reference elements the special globes,  let $\{\gamma_i\}_{i\in\Z}$ be the collection of special globes,  where the special globes are ordered so that the right endpoint of $\gamma_i$ is less than the right endpoint of $\gamma_j$ if $i < j$, and we assume that $\gamma_0$ has the largest non-positive right endpoint.  Let $c_i$ be the center of the special globe $\gamma_i = [c_i-1, c_i +1]$ and $s_i$ the unique point of $\Pi$ in $\gamma_i$.  The following  application of Proposition \ref{key} will be used in the proof of Proposition \ref{inter}.

	\begin{corollary}
		\label{PointUSE}
	Let $\mc S$ be a selection rule.  	Let $\Pi$ be a Poisson point process on $\R$.  Let $\ns{U_i}_{i\in\bb N}$ be a sequence of independent  random variables that are uniformly distributed on $[-1,1]$ and independent of $\Pi$.  Then 
		\begin{equation*}
			\big(\Pi|_{\usg{\Pi}^c},\usg{\Pi}, \sum_{i \in \Z} \delta[s_i]  \big)  \eqd \Big ( \Pi|_{\usg{\Pi}^c},\usg{\Pi},\sum_{i \in \Z} \delta[U_i+c_i] \Big).
		\end{equation*}		
	\end{corollary}

\begin{proof}
	The proof follows from Proposition \ref{key} and elementary properties of Poisson point processes; see \cite{Sw2} for details.  
	\end{proof}

Corollary \ref{PointUSE} tells us that the special points are uniformly distributed and in terms of distribution and  can be treated as is if they are generated using a  random source that is independent of $\Pi$, so that using the randomness from the special points is the same in distribution as using randomness from an independent source.

\subsubsection{Borel isomorphisms}  The following lemma which is a generalization of the mappings $f_k$ in Example \ref{trans-ex}  will allow us to generate from a single uniform random variable, a uniform random variable and a number of i.i.d.\ random variables.  

\begin{lemma}
	\label{borel}
	Let $(A, \alpha)$ be a standard probability space.  For each $k \in \Z^{+}$, there exists a measurable bijective  function $b_k:[0,1] \to [0,1] \times A^k$ such that the push-forward of Lebesgue measure under $b_k$ is the product of Lebesgue measure and $\alpha^k$. 
	\end{lemma}
	Lemma \ref{borel}  is an immediate consequence of the Borel isomorphism theorem \cite[Theorem 3.4.23]{borel}. For the purposes of Theorem \ref{new},  we are only interested in the case where $A$ is finite or countable, so  more elementary proofs are available.

\begin{proof}[Proof of Lemma \ref{borel} (in the case that $A$ is countable)]
Assume $$A = \N= \ns{0,1,2, \ldots}$$ and that $\alpha_i >0$ for all $i \in \N$.   Set  $\ell_0 =0$ and let  $\ell_k = \sum_{i=0} ^{k-1} \alpha_i$   for $k >0$.  Consider the partition of $[0,1)$ into the intervals $\ns{ I_k=[\ell_k, \ell_{k+1}) }_{k=0} ^{\infty}.$    In the same spirit as  binary expansions, we will define a mapping $b:[0,1) \to \N^{\N}$.  Let $u \in [0,1)$.  If $u \in I_k$, then take $b(u)_0 = k$.  For the next coordinate, take the index  of the interval for which the renormalized point $ (u - \ell_k)/\alpha_k$ belongs.  By iterating this procedure, we obtain the mapping $b$.  

It is elementary to check that $b$ pushes forward Lebesgue measure to the product measure $\alpha^{\N}$ and is a bijection almost surely.  Hence  we obtain the required mapping  by transforming an element on the unit interval  to a sequence in $A$, saving $k$ terms of the sequence, then transforming the rest of the sequence back to an element on the unit interval:
\begin{equation*} 
	b_k(u) = \big[ b^{-1}   \big( b(u))_{i=k} ^{\infty}\big), b(u)_{i=0} ^{k-1}\big].
\qedhere
	\end{equation*}
	\end{proof}

\subsubsection{Extending Example \ref{trans-ex}}  We are  ready for the proof of Proposition \ref{inter}. 

\begin{proof}[Proof of Proposition \ref{inter}]
Let $(A, \alpha)$ be a standard probability space.    Let $\mc S$ be the selection rule from Proposition \ref{key}.  Let $\Pi$ be a Poisson point process on $\R$ with intensity $r>0$. 
       Let $c_i$ be the center of the special globe $\gamma_i$ and $s_i$ the unique point of $\Pi$ in $\gamma_i$.   Also let $$v_i = (1+s_i -c_i)/2 \in [0,1].$$
  The real line is partitioned into closed intervals of length two given by the special globes, and open intervals  $\gamma'_{i-1}$ of random length which lie between $\gamma_{i-1}$ and $\gamma_{i}$.   Note that by Proposition \ref{key} \eqref{distance}, $\gamma'_{i-1}$ has length at least $50$.     Let $n_i \geq 1$ be {\em one more} than the number of non-special points of $\Pi$ that are in the interval $\gamma'_{i-1}$.  Order the points of $\Pi$ in $\gamma'_{i-1}$ by their distance to $\gamma_i$ , so that the closest point receives the \dff{rank} $1$, and the next point receives the rank $2$, and the 
  furthest 
  point receives the rank $n_i-1$.      We note that 
  given $\ug{\Pi}$, 
   the numbers $n_k$ and the ranks depend only on $\Pi \ronn_{\usg{\Pi}^c}.$   Let $\ns{U_i}_{i\in\bb N}$ be a sequence of independent  random variables that are uniformly distributed on $[0,1]$  and independent of $\Pi$.  Let $b_k$ be the isomorphisms from Lemma \ref{borel}.    Write $b_k(u) = [ g_k(u), m_k(u)] \in [0,1] \times A^{n_k}$.     By Corollary \ref{PointUSE}, we have
\begin{equation}
	\label{applied}
	\big(\Pi|_{\usg{\Pi}^c},\usg{\Pi}, \ns{ b_{n_i}(v_i) }_{i \in \Z}   \big)   \eqd 
			  \big ( \Pi|_{\usg{\Pi}^c},\usg{\Pi}, \ns{ b_{n_i}(U_i) }_{i \in \Z} \big).
\end{equation}	

    From \eqref{applied}, and Corollary \ref{PointUSE}, we have
$$ \Pi':= \Pi - \sum_{i \in \Z} \delta[s_i] + \sum_{i \in \Z} \delta\big[ 2 g_{n_i}(v_i)-1 + c_i\big] \eqd \Pi;$$
furthermore, from the fact that the coordinates of $b_k$ are independent, we have   
\begin{equation}
	\label{marks}
 (\Pi', \ns{m_{n_i} (v_i)}_{i\in \Z} ) \eqd (\Pi, \ns{ m_{n_i}(U_i)}_{i \in \Z }).
 \end{equation}
Note that $\Pi$ agrees with $\Pi'$ everywhere except on the special globes; in addition, $\Pi'$ and $\Pi'$ have the same special globes, so that the numbers $n_k$ are the ranks of non-special points are the same for both processes.
 
We assign each point of $\Pi'$ a mark in the 
following
way.    If $p \in [\Pi]$, and for some $i \in \Z$, we have  $p \in \gamma_{i-1}'$ with rank $1\leq k < n_i$, then it receives the mark $[m_{n_i}(v_i)]_k$; otherwise for  some $j \in \Z$, we have $p \in \gamma_j$ is a special point, then it receives the mark $[m_{n_j}(v_j)]_{n_j}$.   If $[\Pi'] = \ns{p_i}_{i \in\Z}$, and $X = (X_i)_{i \in \Z}$ are the corresponding marks, then by \eqref{marks}, the law of $\psi(\Pi) := (\Pi', X)$ is $P_r \times \alpha^{\Z}.$ 

By definition, $\psi$ is translation-equivariant and injective $P_r$-almost surely.   It remains to argue that $\psi$ is finitary with a finitary inverse; this fact   follows from Proposition \ref{key}    \eqref{finitary-globes}.  To show that $\psi$ is finitary it  suffices to show that we can determine  $\Pi'$ and the marks the points receive  on $\gamma'_{-1} \cup \gamma_0 \cup \gamma'_0 \cup \gamma_1$ given $\Pi$ restricted to a large finite window.

 Recall that by definition, $\Pi$ agrees with $\Pi'$ except on the special globes; on a special  globe $\gamma_i$  the special point is  resampled and some of the randomness is distributed as marks to the new special point, and the points in the interval $\gamma'_{i-1}$.  Thus $\Pi'$ restricted to $\gamma'_{-1} \cup \gamma_0 \cup \gamma'_0 \cup \gamma_1$ depends only on  $\Pi$ restricted to the union of  $\gamma'_{-1}, \gamma_0, \gamma'_0, \gamma_1.$
  By Proposition \ref{key} \eqref{finitary-globes}, we can determine the intervals  $ \gamma'_{-1}, \gamma_0, \gamma'_0, \gamma_1$ and points of $\Pi$ on these intervals, given $\Pi$ restricted to an even larger interval  $[a-200, b+200]$, where $a$ is left endpoint of $\gamma_{-2}$ and $b$ is the right endpoint of $\gamma_1$.   Similarly, we have that   $\psi^{-1}$ is finitary.   
 \end{proof}
 
 \section{Concluding remarks}

A general measurable isomorphism of Bernoulli shifts  $\phi: A^{\Z} \to A^{\Z}$ may be difficult to understand and  not be effectively realizable by computers, since $\phi$ may require   one to examine a complete bi-infinite sequence $x \in A^{\Z}$ just to determine the output of one symbol $\phi(x)_0$.  
We hope to prove (without using the Ornstein existence criteria) finitary counterparts for some of the systems that have been proven to be Bernoulli and do for  other systems what Keane and Smorodinsky did for Bernoulli shifts and Markov chains.         In particular, 
  we hope to be able to drop the assumption in Theorem \ref{new} that the holding rates are uniform.

We remark that in our proof of Theorem \ref{new} the assumption of a uniform holding rate is used in  two crucial ways.  Firstly, the assumption implies  that the jump times of the Markov chain are given by a Poisson point process.  Secondly, the assumption gives that the skeleton is independent of the jump times.  These two ingredients are lost without the uniform holding rate assumption. 

We also note that the assumption that the skeleton is mixing is crucial in our application of Keane and Smorodinsky's finitary isomorphism to code the states into i.i.d.\ marks.  Regarding the possibility of a periodic skeleton, we ask the following simple question.

\begin{question}
	\label{colour}
	 Does there exists a finitary isomorphism taking a Poisson point process on $\R$ to a marked Poisson point process on $\R$ where the points are coloured red and blue in succession?
	\end{question}

The answer to Question \ref{colour}  is {\em yes} if the requirement that map be finitary is dropped using  Theorem \ref{old} applied to a Markov chain with two states and a uniform holding rate.    If the answer to Question \ref{colour} is {\em no}, then the mixing condition on the skeleton in Theorem \ref{new} cannot be simply removed.  

Note that Holroyd, Pemantle, Peres, and Schramm  prove that one cannot as a factor of a Poisson point process produce a colouring of the {\em same} Poisson point process, where the points are coloured red and blue in succession   \cite[Lemma 11]{random}.

  \subsection*{Acknowledgements}
  
  I thank  Zemer Kosloff and Amanda Wilkens for their  helpful conversations.    I also  thank the referee for carefully reviewing this article and providing insightful suggestions and comments.

   	\bibliographystyle{abbrv}
   	\bibliography{embedding}

\end{document}